\documentclass[a4paper,12pt]{article}

\usepackage{ucs}
\usepackage[utf8x]{inputenc}
\usepackage{amsfonts} 
\usepackage{amsmath} 
\usepackage[dvips]{graphicx}
\usepackage{caption} 
\usepackage{subcaption} 
\usepackage{sidecap} 
\usepackage{enumerate}
\usepackage{setspace}
\usepackage{MnSymbol}
\usepackage[toc,page]{appendix}
\usepackage{color}
\usepackage{soul}

\title{}
\author{}
\date{}

\newtheorem{theorem}{Theorem}[section]
\newtheorem{lemma}[theorem]{Lemma}

\newtheorem{corollary}[theorem]{Corollary}

\newenvironment{proof}[1][Proof]{\begin{trivlist}
\item[\hskip \labelsep {\bfseries #1}]}{\end{trivlist}}

\newenvironment{@abssec}[1]{%
\if@twocolumn
\section*{#1}%
\else
\vspace{.05in}\footnotesize
\parindent .2in
{\upshape\bfseries #1. }\ignorespaces
\fi}
{\if@twocolumn\else\par\vspace{.1in}\fi}

\newcommand{\qed}{\nobreak \ifvmode \relax \else
\ifdim\lastskip<1.5em \hskip-\lastskip
\hskip1.5em plus0em minus0.5em \fi \nobreak
\vrule height0.75em width0.5em depth0.25em\fi}

\newcommand{\torus}{{\mathbb{T}^d}}
\newcommand{\Torus}{{\mathbb{T}^2}}

\newcommand{\mm}{\mathbb}
\newcommand{\Pow}{m}
\newcommand{\Iota}{i}
\newcommand{\Q}{\mbox{WB}}
\newcommand{\Trho}{T_{ \rho }}
\newcommand{\Poincare}{Poincar{\'e}}

\newcommand\keywordsname{Key words}
\newcommand\AMSname{AMS subject classifications}

\begin{document}
\title{Super convergence of ergodic averages for quasiperiodic orbits}
\author{Suddhasattwa Das\footnotemark[1], \and James A Yorke\footnotemark[2]}
\footnotetext[1]{Department of Mathematics, University of Maryland, College Park}
\footnotetext[2]{Institute for Physical Science and Technology, University of Maryland, College Park}

\date{\today}
\maketitle

\renewcommand*\contentsname{Contents}

\begin{abstract}
The Birkhoff Ergodic Theorem asserts that time averages of a function evaluated along a trajectory of length $N$ converge to the space average, the integral of $f$, as $N\to\infty$, for ergodic dynamical systems. But that convergence can be slow. Instead of uniform averages that assign equal weights to points along the trajectory, we use an average with a non-uniform distribution of weights, weighing the early and late points of the trajectory much less than those near the midpoint $N/2$. We show that in quasiperiodic dynamical systems, our weighted averages converge far faster provided $f$ is sufficiently differentiable.  This result can be applied to obtain efficient numerical computation of rotation numbers, invariant densities and conjugacies of quasiperiodic systems.
\end{abstract}


\section{Introduction}
\textbf{Quasiperiodicity.} Each $ \rho \in (0,1)^d$ defines a (pure) \textbf{rotation} on the $d$-dimensional torus $\torus$, i.e. the map $\Trho$ defined as
\begin{equation}\label{eqn:Rot}
\Trho: \theta \mapsto \theta + \rho \mod~1\mbox{ in each coordinate.}
\end{equation}
This map acts on each coordinate $\theta_j$ by rotating it by some angle $\rho_j$.
We call the $\rho_j$'s {\bf ``rotation numbers''}.
We say $ \rho =(\rho_1,\ldots,\rho_d)\in\mathbb{R}^d$ is  \textbf{irrational} if there are no integers $k_j$ for which $k_1 \rho_1 + \cdots + k_d \rho_d \in\mathbb{Z}$, except when all $k_j$ are zero. In particular, this implies that each $\rho_j$ must be irrational. 
The simplest example of $d$-dimensional quasiperiodicity is a (pure) \textbf{irrational rotation}, that is, a rotation by an irrational vector $\rho$.
A map $T:X\to X$ is said to be \textbf{$d$-dimensionally $C^m$ quasiperiodic} on a set $X_0\subseteq X$ for some $d\in\mathbb{N}$ iff there is a $C^m$-diffeomorphism $h:\torus\rightarrow X_0$, such that,
\begin{equation}\label{eqn:conjugacy2}
T(h(\theta)) = h( \Trho ( \theta)).
\end{equation}
where $\Trho$ is an irrational rotation. In this case, $h$ is a conjugacy of $T$ to $\Trho$. It has been conjectured in \cite{sander:yorke:15} that for typical, randomly chosen physical systems, it is one of apparently only three kinds of maximal recurrent sets that are likely to be present, the other two being periodic orbits and chaotic sets.

{\bf An example of a quasiperiodic trajectory of a flow.} To see quasiperiodicity, one needs only look at the moon. See \cite{moon_orbit}. More precisely, there are quasiperiodic trajectories when modeling the sun-earth-moon system as a circular restricted three-body problem in $\mm R^6$ (position + velocity) with Newtonian gravitation. Let $q_1, q_2, q_3$ be the positional coordinates where earth-sun rotation is in the $q_1-q_2$ plane. The moon is not restricted to that plane. Using rotating coordinates in which the sun and earth are fixed on the $q_1$ axis, we can get an excellent approximation of the moon's orbit as a quasiperiodic flow on $T^3 \subset \mm R^6$. ``Restricted’’ means that the moon's mass is modeled as negligible. Note that the masses are modeled as points so tidal forces are not included, nor is the influence of the other planets. 

{\bf An example of a quasiperiodic trajectory of a map.} Following \Poincare, we also can get the Poincar{\'e}  return map $F$ by fixing a value for the Hamiltonian of the system and by taking trajectory points on the $q_1-q_3$ plane, where $q_2=0$ and $dq_2/dt >0$. The moon passes through this plane about once a month. Here the flow on $\mm T^3$ is reduced to a map on $\mm T^2$. We note that the map $F$ is real analytic.

{\bf The Conjugacy Problem}. Given an apparently $d$-dimensional chaotic trajectory $(x_n)$ in $R^D$, how do we demonstrate that it is quasiperiodic? From the definition, it is sufficient to find the conjugacy map $h:T^d\to R^D$ and find the rotation numbers $ \rho  = (\rho_1, \cdots, \rho_d)$ for which 
$x_n = h(n \rho_1, \cdots, n\rho_d).$ 
If $h$ is found, then we will have shown that the pure rotation with rotation numbers $ \rho $ is conjugate to $F$ restricted to $X_0 = h(T^d)$. We call the problem of finding $ \rho $ and $h$ -- knowing only the trajectory $(x_n)$ -- the {\bf Conjugacy~Problem} for a quasiperiodic trajectory. Our approach to finding $h$ is to compute the Fourier Series for $h$ using $\Q_N$. 

\textbf{The Birkhoff Ergodic Theorem.} Let $T:X\rightarrow X$ be a map on a topological space $X$ with a probability measure $\mu$ for which $T$ is invariant. Given a point $x$ in $X$ and a real- or vector-valued function $f$ on $X$, we will refer to a long-time average of the form
\begin{equation}\label{eqn:Birkhoff_sum}
B_N(f) :=
 \frac{1}{N}\sum\limits_{n=0}^{N-1}f(T^n(x)),
\end{equation}
as a {\bf Birkhoff average}.
Throughout this paper, $f:X \to E$ where $E$ is a finite-dimensional real vector space.
The limits of such sequences frequently occur in dynamical systems.
If $\mu$ is a probability measure on $X$ and $T$ preserves the probability measure $\mu$ and is ergodic, then the von Neumann Ergodic Theorem (e.g. see Theorem 4.5.2. in \cite{BrinStuck}) states that { for $f\in L^2(X,
\mu)$, the Birkhoff average (Eq. \ref{eqn:Birkhoff_sum}) converges in the $L^2$ norm to the integral $\int_X fd\mu$.
}
The Birkhoff Ergodic Theorem (see Theorem 4.5.5. in \cite{BrinStuck}) strengthens von Neumann's theorem and concludes that if $f\in L^1(X,
\mu)$, then Eq. \ref{eqn:Birkhoff_sum} converges to the integral $\int_X fd\mu$ for $\mu$-a.e. point $x\in X$.
The Birkhoff average (Eq. \ref{eqn:Birkhoff_sum}) can be interpreted as an approximation to an integral, but convergence is very slow. For any non-constant $f$, there is a constant $C>0$ such that the following holds for infinitely many $N$.
\[|\frac{1}{N}{\Sigma}_{n=1}^Nf(T^n(x))- \int_X fd\mu| \ge C N^{-1}.\]

For general ergodic dynamical systems, the rate of convergence of these sums can be arbitrarily slow, as shown in \cite{ErgCnvrgnc2}.
For many purposes the speed of convergence is irrelevant but it is important for numerical computations.

{\bf Definition.}
Let $(a_N)_{N=0}^\infty$ be a sequence in a normed vector space such that $a_N \to b$ as $N\to\infty$. We say $(a_N)$ has {\bf super-polynomial convergence} to $b$ or {\bf super converges} to $b$ if for each integer $m>0$ there is a constant $C_m >0$ such that
\[|a_N - b| \le C_m N^{-m} \mbox{ for all m}.\]

{\textbf{Definition.} A function $w:\mathbb{R}\to [0,\infty)$ is said to be a \boldmath $C^\infty$ \unboldmath \textbf{bump function} if $w$ is $C^\infty$ and
 the support of $w$ is $[0,1]$ and $\int_\mathbb{R}w(x)dx \ne 0$ and $w$ and all of its derivatives vanish at $0$ and $1$. When we use a bump function as in Eq. \ref{eqn:QN}, it is independent of scaling; hence we can assume without loss of generality that $\int_\mathbb{R}w(x)dx=1$.
 Eq. \ref{eqn:weight} below gives an example of a $C^\infty$ bump function that we will call the {\bf exponential weighting},

\begin{equation}\label{eqn:weight}
w(t)=\begin{cases}
\exp\left(\frac{1}{t(t-1)}\right), & \mbox{for } t\in(0,1)\\
0, & \mbox{for } t\notin(0,1)
\end{cases}
\end{equation}

Instead of weighting the terms $f(T^n(x))$ in the average equally, we weight the early and late terms of the set ${1,\cdots,N}$ much less than the terms with $n\sim N/2$ in the middle.
We insert a weighting function $w$ into the Birkhoff average to get what we call a {\bf Weighted Birkhoff ($\Q_N$) average}. Let \boldmath $\torus$ \unboldmath denote a $d$-dimensional torus. For $X = \torus$ and a continuous $f$ and for $\phi\in\torus$, we define 
\begin{equation}\label{eqn:QN}
\Q_N(f)(x) :=\frac{1}{A_N}\sum\limits_{n=0}^{N-1}w(\frac{n}{N})f(T^{n}x),\mbox{ where }A_N:= \sum\limits_{n=0}^{N-1}w(\frac{n}{N}).
\end{equation}
Of course the sum of the terms $w(\frac{n}{N})/A_N$ is $1$.
Continuity of $w$ guarantees that $\underset{N\rightarrow\infty}{\lim}A_N/N=\int_0^1 w(t)dt$. We introduced the {\bf exponentially weighted Birkhoff average} of Eq. \ref{eqn:QN} for numerical investigations of quasiperiodic systems in \cite{DSSY}. Motivated by finding how effective our method was numerically, we discovered a simple proof of its super convergence.

\textbf{Unique invariant probability measure.} Lebesgue probability measure \boldmath $\lambda$ \unboldmath on $\torus$ is the unique invariant probability measure for any irrational rotation $\Trho$. A map $T:X_0\rightarrow X_0$ therefore has a unique invariant probability measure $\mu$ which is defined as \[\mu(U):=\lambda(h^{-1}(U)),\mbox{ for every Borel set }U\subseteq X_0.\]

{\bf Diophantine rotations.} An irrational vector $  \rho \in\mathbb{R}^d$ is said to be {\bf Diophantine} if for some $\beta>0$ it is {\bf Diophantine of class} $\beta$ (see \cite{HermanSeminal}, Definition 3.1), which means there exists $C_\beta>0$ such that for every $ k \in\mathbb{Z}^d$, $ k \neq 0$ and every $n\in\mathbb{Z}$,
\begin{equation}\label{eqn:Dioph}
| k \cdot \rho -n|\geq\frac{C_\beta}{\|k\|^{d+\beta}.}
\end{equation}
For every $\beta>0$ the set of Diophantine vectors of class $\beta$ have full Lebesgue measure in $\mathbb{R}^d$ (see \cite{HermanSeminal}, 4.1). The Diophantine class is crucial in the study of quasiperiodic behavior, for example in \cite{Simo1} and \cite{Simo2}. Ineq. \ref{eqn:Dioph} is central to Arnold's study of small denominators in \cite{Arnold}.

\begin{theorem}\label{thm:A1}
Let $X$ be a $C^\infty$ manifold and $T:X \to X$ be a d-dimensional $C^\infty$ map which is quasiperiodic on $X_0\subseteq X$, with invariant probability measure $\mu$. Assume $T$ has a Diophantine rotation vector. Let $f:X \to E$ be $C^\infty$, where $E$ is a finite-dimensional, real vector space. Assume $w$ is a $C^\infty$ bump function. Then for each $x_0\in X_0$, the weighted Birkhoff average
$\Q_Nf(x_0)$ has super convergence to $\lim_{N\to\infty} B_N(f)(x_0) = \int_{X_0}fd\mu$. Moreover, the convergence is uniform in $x_0$.
\end{theorem}
Theorem \ref{thm:A1} is proved as a corollary of Theorem \ref{thm:A2} in Section \ref{sect:Cm_Version}. 

\textbf{Applicability to chaos?} Since $\Q_N(f)(x_0)$ and $B_N(f)(x_0)$ have the same limit as $N\to\infty$, the weighted Birkhoff average provides a fast way of computing the limit of a Birkhoff average. If the ergodic process $T$ is chaotic instead of quasiperiodic, then weighted Birkhoff averages provide no advantage over Birkhoff averages.

\textbf{The need for fast convergence.} We note that in \cite{DSSY} we get convergence to the limit of our 30-digit numerical precision in all of our 1D quasiperiodicity studies well before $N=10^6$ iterates. 
If we used $B_N$, we would expect to need about $N\sim 10^{30}$, higher by a factor of about $10^{25}$. 
Hence if we can compute $10^6$ iterates per second, $\Q_N$ would require one second while $B_N$ would require one billion billion years. 

\textbf{Related methods.} The accelerated convergence of weighted ergodic sums have been discussed in \cite{ErgAvg4} , \cite{ErgAvg6} and \cite{ErgAvg7}, but without conclusions of specific rates of convergence. In \cite{ErgAvg8}, a convergence rate of $O(N^{-\alpha})$, $(0<\alpha<1)$, was obtained for functionals in $L^{2+\epsilon}$ for a certain choice of weights.

In \cite{luque:villanueva:14}, A. Luque and J. Villanueva develop methods for obtaining rotation numbers by taking repeated averages of averages of a quasiperiodic signal. By taking $p$ nested averages, their method converges to the rotation number with an error bounded by $C_p N^{-p}$, where $C_p$ is a constant and $N\geq 2^p$. Their method of computation depends on the choice of convergence rate $p$ and as $p$ increases the computational complexity increases for fixed $N$. Also, their computation time $T(p,N)$ obeys $T(p,N)/N\rightarrow\infty$ as $p\rightarrow\infty$. In comparison, computation time for our weighted Birkhoff average is simply proportional to $N$ since it requires an average of $N$ numbers, including the evaluation of $w$ at $N$ points. 

Jacques Laskar uses a $C^1$ weighting function, the Hanning window filter, essentially the function $\sin^2(\pi t)$ on $(0,1)$ except that he uses a two-sided version on $(-1,1)$, in investigating almost periodic perturbations of almost periodic flows in several papers \cite{Laskar99,Laskar93a,Laskar93b,Laskar03}.  But he also states in an appendix (Remark 2, p.146) to \cite{Laskar99}) that one can consider what is essentially our Weighted Birkhoff average, Eq. \ref{eqn:QN}. He reports that it has excellent asymptotic properties, 
but we do not find in his work any examples where he has implemented it.

{\bf A companion paper.} This project has two components. The rigorous foundation for our methods are reported here and the numerical work is reported in \cite{DSSY}; the effectiveness of the theoretical work is demonstrated through the numerical results; that is, we get really fast convergence in practice. The reliability of the numerical work in \cite{DSSY} is correspondingly  based on the rigorous results presented here so that when we know a sequence is guaranteed to converge and then we get numerical convergence to the available precision, it suggest strongly that the results are correct. In that paper we present four examples:
\begin{itemize}
	\item The restricted circular three-body problem in $\mm R^2$ with 1D quasiperiodic curves (dimension d=1).
	\item A periodically forced van der Pol oscillator with an attracting quasiperiodic curve.
	\item The Chirikov-Taylor standard map (area-preserving map on $\Torus$) with quasiperiodic closed curves.
	\item A map on a torus $\Torus$ exhibiting $d=2$-dimensional quasiperiodicity.
\end{itemize}

The first three of these investigations are for one-dimensional quasiperiodic curves in dimension 2, while the fourth is for two-dimensional quasiperiodic dynamics on $\Torus$. All of these problems are real analytic, so our hypotheses below of $C^\infty$ dynamics is satisfied by these examples.
In addition, toy examples are presented in which a rotation number is known \emph{a priori} and the methods converge to the actual rotation number. 

\section{Applications of Theorem \ref{thm:A1}}


\subsection{Example 1. Rotation vectors}
By the Birkhoff ergodic theorem, a function can be integrated using just its values along a trajectory, without knowing the points on the trajectory or the invariant set is. Given a map $T:X\to X$ with a quasiperiodic set $X_0$ and a trajectory $(x_n)\subset X_0$,we must convert each $x_n$ into a set of $d$-dimensional angles $\phi_n$, as shown for a 1D-case in Fig. \ref{fig:Cashew}. Our goal is to determine the rotation vector $ \rho \in\torus$, purely from $(\phi_n)\in\torus$. There are some hidden subtleties. The standard approach towards computing $ \rho $ is as the limit $ \rho =\underset{N\rightarrow\infty}{\lim}N^{-1}(\bar\phi_N-\bar\phi_0)$, where the quantity $\bar\phi_N-\bar\phi_0$ is computed as $\sum\limits_{n=0}^{N-1}\Delta\phi_n$, where $\Delta\phi_n:=\angle(\phi_{n+1},\phi_n)$ is some measure of the angle from $\phi_n$ to $\phi_{n+1}$. 

Given a quasiperiodic trajectory, we need to define smooth functions $\Delta\phi$ and $\phi$, taking values in $\mathbb{R}^d$ and $\mathbb{T}^d$ respectively. This general problem can be quite difficult.

Even in the 1D-case, there are two ways to measure the angle between $\phi_n$ and $\phi_{n+1}$. Sometimes, $\Delta\phi$ can be chosen to be the positive angle from $\phi_n$ to $\phi_{n+1}$. That would not work however in Fig. \ref{fig:Cashew}, where $\Delta\phi$ must be the shortest angular difference from $\phi_n$ to $\phi_{n+1}$. In another example, if an astronomer measures the angular position of Mars from the earth, the angle sometimes makes small positive changes and sometimes negative, exhibiting ``retrograde motion''. In such a case, it is appropriate to choose $\Delta\phi$ so as to minimize the absolute value of the angular change. 

To avoid these problems, we will make two assumptions on $\phi$ and $\Delta\phi$. We begin with a definition.

\textbf{Definition. } Let $\phi:X_0\to \torus$ be a $C^\infty$ map. Then the lift $\bar\phi:\bar{X_0}\equiv\mathbb{R}^d\to\mathbb{R}^d$ can be written in the form $\bar\phi(\theta) = A\bar\theta + g(\theta) $, where $g : \mathbb{R}^d\to \mathbb{R}^d$ is bounded (and periodic with period 1 in each coordinate), and $A$ is a matrix called the \textbf{homology matrix} of $\phi$, so that
\[A = \phi_*:H_1(X_0)\to H_1(\torus).\]
Note that the entries of $A$ are again integers. We say $\phi$ is a {\bf unit degree map} if the determinant of $A$ is $\pm 1$. In that case, $A^{-1}$ is also an integer valued matrix with determinant $\pm 1$. Note that, if $\phi$ is unit-degree, then there is a choice of coordinates for $\theta$ under which $A$ is the identity, and henceforth, we will assume that $A$ is the identity.

\textbf{(A1) }$\phi$ is a $C^\infty$ unit degree map.

\textbf{(A2) }$\Delta\phi:X_0\to\mathbb{R}^d$ is a $C^\infty$ map such that for each $x\in X_0$, if $\bar{\phi(x)}$ is a lift of $\phi(x)$, then $\Delta\phi(x)+\bar{\phi(x)}$ is a lift of $\phi(T(x))$.

 Let $\phi_n$ be the point $\phi(x_n)$. Then note that $\Delta\phi(x_n)$ is a lift of $\phi_{n+1}-\phi_{n}$. Fix any lift $\bar\phi_0$ of $\phi_0$ and inductively define \boldmath $\bar\phi_n$ \unboldmath $:=\bar\phi_{n-1}+\Delta\phi(\theta_{n-1})$. Note that by virtue of $\Delta\phi$, $\bar\phi_n$ is always a lift of $\phi_n$.

\begin{corollary}\label{prop:rot_num_convergence}
Let $X,X_0,T,$ and $w$ be as in Theorem \ref{thm:A1}. Let $\phi: X_0\to \torus$ and $\Delta\phi:X_0\to\mathbb{R}^d$ satisfy assumptions (A1) and (A2). 
Then for every initial point $x_0\in X_0$,
\[\Q_N(\Delta\phi) :=\frac{1}{A_N}\sum \limits_{n=0}^{N-1}w(\frac{n}{N})\Delta\phi_n\]
has super convergence to a vector $\bar{ \rho }$, which is a lift of the rotation vector $ \rho $, uniformly in $x_0$ as $N \to\infty$. 
\end{corollary}
In other words, $ \rho =\bar{ \rho }\mod~1$, (i.e., mod 1 in each coordinate).
\begin{proof} By Theorem \ref{thm:A1}, the sum has super convergence to the integral $\int_{X_0}\Delta\phi d\mu$. By the Birkhoff Ergodic Theorem, $\int_{X_0}\Delta\phi d\mu$ is also the limit of the unweighted Birkhoff sum $\frac{1}{N}\sum \limits_{n=0}^{N-1}\Delta\Phi_n$, which is equal to $\frac{1}{N}\sum \limits_{n=0}^{N-1}[\bar{\phi}_{n+1}-\bar{\phi}_n]$, which is equal to $(\bar{\phi}_N - \bar\phi_0)/N$. Therefore, the sum in the claim has super convergence to the limit $\underset{N\rightarrow\infty}{\lim}\frac{\bar{\phi}_N - \bar\phi_0}{N}$. We will now show that this limit is the rotation number $ \rho $.

Note that since $(x_n)$ is a quasiperiodic trajectory, there exists an unknown rotation number $ \rho $ and an unknown continuous, periodic function $h:\mathbb{R}^d\to\mathbb{R}^d$ so that for each $n$, $\bar x_n=n \rho +h(n \rho )$ is a lift of $x_n$. 
\\Therefore $\bar{\phi}_n=A\bar x_n+g(\bar x_n)=nA \rho +Ah(n \rho )+g(n \rho +h(n \rho ))$.
\\Therefore $\frac{\bar{\phi}_N}{N}=A \rho +\frac{Ah(n \rho )}{N}+\frac{g(n\rho+h(n \rho ))}{N}$, which tends to $A \rho $ as $N\rightarrow\infty$.
\\Therefore, $\underset{N\rightarrow\infty}{\lim}\frac{\bar{\phi}_{N-1}-\bar{\phi}_0}{N}=A \rho $. \qed
\end{proof}

\textbf{Remark.} Consider the case of a $1$-dimensional quasiperiodic set $X_0$ embedded in $X=\mathbb{R}^2$. Let $C:= C_B \cup C_U$ be the complement of $X_0$ in $\mathbb{R}^2$, where $C_B$ and $C_U$ are the bounded and unbounded components of $C$ respectively. For $p\in \mathbb{R}^2$, define \begin{equation}\label{eqn:phi_1D}
\phi(x) = (x - p)/\|x - p\|.
\end{equation}
Hence $\phi(\theta)\in\torus$ where $d=1$. If $p\in C_B$, then $\phi$ is a unit degree map. Suppose that the range of values of $\phi(T(x))-\phi(x)\mod 1$ does not include an angle $\theta_0$. Then $\Delta\phi$ can be taken to be the difference
\[
\Delta\phi(x):\phi(T(x))-\phi(x)-\theta_0\pmod{1}+\theta_0
\]
So $\Delta\phi_n\in[0,1)$ is the length of the arc joining $\phi_n$ to $\phi_{n+1}$ and avoiding the angle $\theta_0$.. See Fig. \ref{fig:Cashew}. 

\begin{figure}
\centering
\includegraphics[scale=0.2]{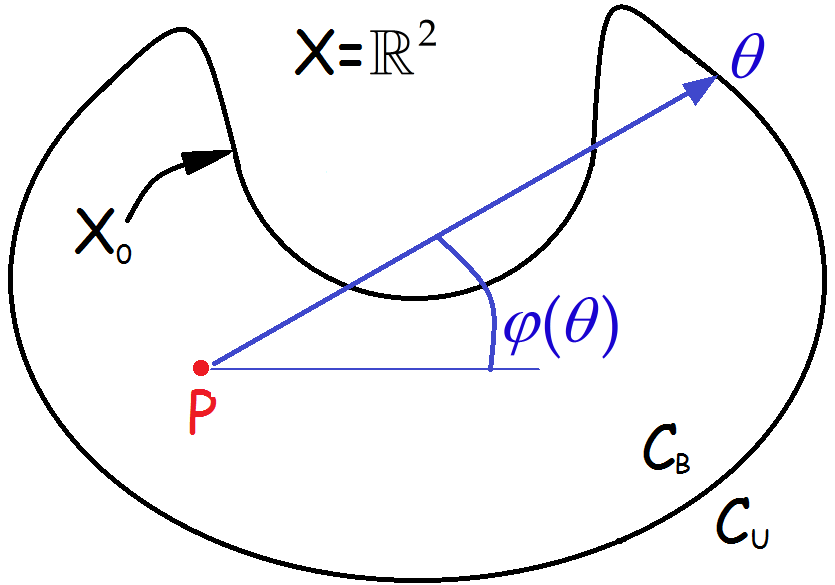}
\caption{\textbf{Rotation number on a quasiperiodic curve.} Given a quasiperiodic curve $X_0$ embedded in $X=\mathbb{R}^2$, one can define $\phi$ as in Eq. \ref{eqn:phi_1D}. Let $\rho_\phi := \lim_{N\to\infty}\frac{1}{N}\sum \limits_{n=0}^{N-1}w(\frac{n}{N})[\bar{\phi}_{n+1}-\bar{\phi}_n]$. If $p$ lies outside the curve $X_0$, then $\rho_\phi=0$. If $p$ lies in the interior of $X_0$, then $\rho_\phi $ is $\rho$ or $1-\rho$, both being legitimate representations of $\rho$. Note that $\bar{\phi}_{n+1}-\bar{\phi}_n$ can assume both positive and negaitve values.}
\label{fig:Cashew}
\end{figure}

\subsection{Computing Fourier series.}

{\bf Example 2, Fourier Series of the Embedding.}
Assume $X\subset \mathbb{R}^D$.
If a map $T:X\rightarrow X$ is quasiperiodic on $X_0$,
then $X_0$ has
a parameterization $x = h(\theta)$, where $h:\torus \to \mathbb{R}^D$, for which $x_{n+1}=T(x_n)$ is conjugate to the pure rotation $\theta_{n+1}=\theta_n+ \rho $.

We can compute a Fourier series for $h(\theta)$ provided the rotation number $ \rho $ for $\theta$ is known, but that can be calculated as a weighted average, as explained above in Corollary \ref{prop:rot_num_convergence}.
The map $h$ is not known explicitly, but its values ${(x_n:=h(n \rho \mod~1))_{n=0,1,2,\ldots}}$ are known.
For every $ k \in\mathbb{Z}^d$, the $ k $-th Fourier coefficient of $h$ is described below.
\[a_{ k }(h) := \int_\torus h(\theta)e^{-\Iota 2\pi  k \cdot\theta}d\theta , \mbox{ so } h(\theta)=\underset{ k \in\mathbb{Z}^d}{\Sigma}a_{ k }e^{\Iota 2\pi  k \cdot\theta}.\]
The degree of differentiability of $h$ can be estimated from the decay rate of  $a_{ k }$ with $\| k \|$.
Therefore, the crucial step in this estimation is an accurate calculation of each $a_{ k }$, which can be obtained as the limit of a weighted Birkhoff average, as stated in the following Corollary.

\begin{corollary}\label{cor:Fourier}
Under the assumptions of Corollary \ref{prop:rot_num_convergence}, for each $k\in\mathbb{Z}^d$, the quantity
\[
\Q_N[h(\theta)e^{-\Iota 2\pi  k \cdot\theta}]=
\frac{1}{A_N}\sum\limits_{n=0}^{N-1}w(\frac{n}{N})x_ne^{-\Iota 2\pi n k \cdot \rho }\]
has super convergence as $N\rightarrow\infty$ to $a_{ k }(h)$, the $k$-th Fourier coefficient of $h$
\end{corollary}
\begin{proof}
By Theorem \ref{thm:A1}, the limit of the weighted Birkhoff average of the smooth function $h(\theta)e^{-\Iota 2\pi  k \cdot\theta}$ has super convergence to the integral $\int_\torus h(\theta)e^{-\Iota 2\pi  k \cdot\theta}d\theta$, which is precisely $a_{ k }(h)$ , the $ k $-th Fourier coefficient of $h$. \qed
\end{proof}

Determining the \emph{smoothness of conjugacy to a pure rotation} is an old problem. While we can only determine the degree of differentiability of the conjugacy function $h$ computationally (non-rigorously) by observing how quickly its Fourier series coefficients $a_k$ go to $0$ as $\|k\|\rightarrow\infty$, the papers \cite{HermanSeminal}, \cite{AbsContConj}, \cite{KatzOrn2} and \cite{Yoccoz1} arrive at rigorous conclusions on the differentiability of $h$ by making various assumptions on the smoothness of the quasiperiodic map $T$ and the Diophantine class of its rotation number $\rho$. In our case, we conclude that $h$ is real analytic if $\|a_k\|$ decreases exponentially fast, i.e., $\log \|a_{\vec k}\|\le A + B|\vec k|$ for some $A$ and $B$, to the extent checkable with a given computer precision. 
Also see \cite{DSSY} for a discussion on computing Fourier coefficients of maps between tori instead of maps from the torus to $\mathbb{R}^D$.

{\bf Example 3, computing the integral of a periodic $C^\infty$ function.}
We designed the above theorem as a tool for investigating quasiperiodic sets, but sometimes we can artificially create the quasiperiodic dynamics. To integrate a $C^\infty$ map with respect to the Lebesgue measure when the map is periodic in each of its $d$ variables, we can rescale its domain to a d-dimensional torus $\torus =[0,1]^d\mod~1$ in each coordinate. Choose $ \rho  = (\rho_1,\cdots,\rho_d)\in (0,1)^d$ of Diophantine class $\beta>0$. Let $T = T_{ \rho }$ be the rotation by the Diophantine vector $\rho$ on $\mathbb{T}^d$. Let $w$ be the exponential weighting function Eq. \ref{eqn:weight}. Then by Theorem \ref{thm:A1}, for every $ \theta \in \mathbb{T}^d$, $\Q_N(f)(\theta)$ {\it has super convergence to $\int_{\mathbb{T}^d} f d\mu$ and convergence is uniform in $\theta$}.

\section{A more general $C^m$ version of Theorem \ref{thm:A1}}\label{sect:Cm_Version}

Theorem \ref{thm:A1} is a special case of the following more detailed Theorem \ref{thm:A2}.

\boldmath $C^m$ \unboldmath \textbf{bump functions}. Let $m\geq 0$ be an integer. We will say that a function $w:\mathbb{R}\rightarrow\mathbb{R}$ is a \boldmath $C^m$ \unboldmath \textbf{bump function} if $w$ is a $C^m$ function, is $0$ outside $[0,1]$ and $\int_0^1w(t)dt>0$. Hence $w$ and its first $m$ derivatives are $0$ at $0$ and $1$.
For example, the function which is $x(1-x)$ on $[0,1]$ is $C^0$, while $\sin^2(\pi t)$ is $C^1$. 

\begin{theorem}\label{thm:A2}
Let  $m>1$ be an integer and let $w$ be a $C^{m^*}$ bump function for some $m^* \ge m$. For $M \in\mathbb{N}$, let $X$ be a $C^M$ manifold and $T:X \to X$ be a $C^M$ d-dimensional quasiperiodic map on $X_0\subseteq X$, with invariant probability measure $\mu$ and a rotation vector of Diophantine class $\beta$. Let $f:X \to E$ be $C^{M}$, where $E$ is a finite-dimensional, real vector space.
Then there is a constant $C_m$ depending upon $w,f,m,M,$ and $\beta$ but independent of $x_0\in X_0$ such that
\begin{equation}\label{eqn:Mm2}
\left|(\Q_Nf)(x_0)- \int_{X_0}\!f~d\mu\right|  \le C_m\!N^{-\Pow} ,
\end{equation}
provided the ``smoothness'' $M$ satisfies
\begin{equation}\label{eqn:Mm}
M>d+m(d+\beta).
\end{equation}
\end{theorem}

\textbf{Proof of Theorem \ref{thm:A1} as a corollary of Theorem \ref{thm:A2}.} Assume the hypotheses of Theorem \ref{thm:A1}. Then the hypotheses of Theorem \ref{thm:A2} are satisfied with $M=\infty$.  Therefore, for every integer $m>1$, Ineq. \ref{eqn:Mm2} holds for some $C_m>0$.
Therefore, $\Q_N(f)$ has super convergence to $\int_{X_0}\! fd\mu$, uniformly in $x_0$. \qed

\subsection{Proof of Theorem \ref{thm:A2}}

We need the following lemma in our proof.
\begin{lemma}[Poisson Summation Formula, \cite{PoissonSum}]\label{prop:PoissonSum}
For each $g\in L^2(\mathbb{R})$, 
\[\Sigma_n g(n)=
\Sigma_n \int_\mathbb{R} g(t)e^{-i2\pi nt}dt.\]
\end{lemma}

To prove Theorem \ref{thm:A2}, we use the coordinates $\theta$ on $X_0$ for which Eq. \ref{eqn:conjugacy2} holds. So for the rest of this section, $X_0=\torus$ and $T$ is an irrational rotation $\Trho$, as described in in Eq. \ref{eqn:Rot}. Of course
$T^n_{ \rho }=T_{n \rho }: \theta \mapsto \theta +n \rho \pmod{1}$.
As a result,
$f(T^n_{ \rho }( \theta ))=f( \theta +n \rho )$.
Therefore, in this dynamical system, Eq. \ref{eqn:QN} takes the form
\[
(\Q_N f)(\theta)=\frac{1}{A_N}\underset{n}{\Sigma}w(\frac{n}{N})f( \theta +n \rho ).
\]
For each $f\in L^2(\torus,\lambda)$ we can write its Fourier series representation $f(\theta)=\underset{k\in\mathbb{Z}^d}{\sum} a_k \sigma_k(\theta)$, where $\theta\in\torus$ and \boldmath $\sigma_k(\theta)$ \unboldmath $:=e^{\Iota 2\pi k\cdot\theta}$. Let \boldmath $\mathbb{Z}^d_0$ \unboldmath denote the set of integer-valued vectors excluding the zero vector, i.e., $\mathbb{Z}^d\setminus\{0\}$. Note that the Fourier coefficient $a_0$ of $f$ is the integral $\int_{\torus}fd\theta$ and is also $\Q_N(a_0 \sigma_0)(\theta)$ for every $N$ since $\sigma_0 \equiv 1$. Therefore, 
\[\begin{split}
E_N &:=(\Q_N f)(\theta)-\int_{\torus} f(t)dt\\
&=\Sigma_{k\in\mathbb{Z}^d_0}\ a_k \Q_N \sigma_k(\theta)\\
&= A_N^{-1}\Sigma_{k\in\mathbb{Z}^d_0}a_k \left[\Sigma_{n} \ w(n/N)\sigma_k(\theta+n\rho)\right] \\
&= A_N^{-1}\Sigma_{k\in\mathbb{Z}^d_0}a_k \left[\Sigma_{n} \int_\mathbb{R} w(t/N)\sigma_k(\theta+t\rho)e^{-i2\pi nt}dt \right], \mbox{by Lemma \ref{prop:PoissonSum}}\\
&= A_N^{-1}\Sigma_{k\in\mathbb{Z}^d_0} a_k \left[
e^{i2\pi k\cdot\theta}\Sigma_{n} \int_\mathbb{R} w(t/N)e^{i2\pi t(k\cdot\rho-n)}dt \right]\\
\end{split}\]
Set $s=t/N$ so that $Nds=dt$ and 
\[
\left| E_N \right| \leq (N/A_N)\Sigma_{k\in\mathbb{Z}^d_0}\left| a_k\right|\Sigma_{n}  \left|\int_\mathbb{R} w(s)e^{i2\pi Ns(k\cdot\rho-n)}ds \right|.
\]
But $N/A_N$ is bounded by some constant $C_w>0$ since it converges to $\int_0^1w(s)ds$.

We establish a bound for  
$\int_0^1 w(s)e^{i\Omega s}ds$, where $\Omega :={\Iota2\pi N(k\cdot\rho-n)}$. 
Since $w(s)$ and its first $m$ derivatives vanish at $0$ and $1$, integrating by parts $m$ times yields 
\begin{equation}\label{lem:Fourier_Decay}
\left|\int_0^1 w(s)e^{i\Omega s}ds\right|=\left|\Omega^{-m}\int_0^1 w^{(m)}(s)e^{i\Omega s}ds\right|\le |\Omega|^{-m} \|w^{(m)}\|,
\end{equation}
where $\|w^{(m)}\|$ is the $L^1$ norm of the $m^{th}$ derivative of $w$.

Since $f$ is a $C^{M}$-function, differentiating it $M$ times gives $|a_k|\leq C_{f,M} \|k\|^{-M}$, for each $k\in\mathbb{Z}^d_{0}$ for some \boldmath $C_{f,M}>0$ \unboldmath that depends only on $f$ and $M$.  Ineq. \ref{lem:Fourier_Decay} gives
\[\begin{split}
\left| E_N \right| 
&\leq C_w~\Sigma_{k\in\mathbb{Z}^d_0} ~C_{f,M} \|k\|^{-M}~ \Sigma_{n} ~(2\pi)^{-m} N^{-m} \|k\cdot\rho-n\|^{-m} \|w^{(m)} \|, \\
& = C^* N^{-m}~[\Sigma_{k\in\mathbb{Z}^d_0} ~\|k\|^{-M}~\Sigma_{n}\ \|k\cdot\rho-n\|^{-m}]\\
\end{split}\]
where we have collected constants into $C^* := C_w (2\pi)^{-m}C_{f,M}\|w^{(m)}\|$. Therefore our proof is finished when we show that 
the quantity in the above brackets $B:=\Sigma_{k\in\mathbb{Z}^d_0} \|k\|^{-M}\Sigma_{n}\ \|k\cdot\rho-n\|^{-m}$ is finite. We will bound $B$ by  $\Sigma_{k\in\mathbb{Z}^d_0} \|k\|^{-d^*}$, which is finite if and only if  $d^* > d$.

To bound the sum $\underset{n}{\Sigma}\|k\cdot\rho-n\|^{-m}$, 
note that there is a unique integer $n_0$ which is closest to the number $k\cdot\rho$. 
Ineq. \ref{eqn:Dioph} implies $\|k\cdot\rho-n_0\|^{-m} \leq C_\beta^m |k|^{m(d+\beta)}$. 
The rest of the sum satisfies $\Sigma_{n\neq n_0}\|k\cdot\rho-n\|^{-m}\leq \Sigma_{n\neq 0}(|n|+0.5)^{-m}$, which is bounded since $m>1$ . Hence there is a constant $C_{\rho,m} >1$ such that $\underset{n}{\Sigma}\|k\cdot\rho-n\|^{-m} \leq C_{\rho,m}\| k \|^{m(d+\beta)}$.   
Hence
$B \le constant \times\Sigma_{k\in\mathbb{Z}^d_0} \| k \|^{-d^*}$ where 
$-d^*:=-M+m(d+\beta)$, so by Ineq. \ref{eqn:Mm}, $d^*>d$ and $B$ is finite. \qed

{\bf Acknowledgements} JY was partially supported by National Research Initiative Competitive grants 2009-35205-05209 and 2008-04049 from the U.S.D.A. 
We thank a commentator who wishes to remain anonymous for suggesting the use of the Poisson Summation Formula to simplify our original proof.
\bibliographystyle{unsrt}
\bibliography{Weighted_calc_bibliography,1D_bibliography}
\end{document}